\documentclass[a4paper]{amsart}
\usepackage {graphicx,color}
\usepackage {a4}
\usepackage{amssymb}
\usepackage{amsmath}
\usepackage{amsthm}
\usepackage{mathrsfs}
\usepackage{graphicx,color}
\providecommand{\abs}[1]{\lvert#1\rvert}

\DeclareMathOperator{\area}{area} 
\DeclareMathOperator{\vrt}{vert}  
\DeclareMathOperator{\hol}{hol}   

\newcommand{\Cb}{\mathbb{C}}      
\newcommand{\Rb}{\mathbb{R}}      
\newcommand{\Nb}{\mathbb{N}}      
\newcommand{\Zb}{\mathbb{Z}}      
\newcommand{\Pb}{\mathbb{P}}      

\newcommand{\SL}{\text{SL}(2,\Rb)} 
\newcommand{\GL}{\text{GL}(2,\Rb)} 

\newcommand{\Hf}{\mathcal{H}}     
\newcommand{\Qf}{\mathcal{Q}}     

\newcommand{\Hs}{\mathcal{H}}     
\newcommand{\Ls}{\mathscr{L}}     

\newcommand{\fix}{$\text{2T}_{\text{fix}}\text{2C}$}   

\theoremstyle{plain}
\newtheorem{thm}{Theorem}[section]
\newtheorem{prop}[thm]{Proposition}
\newtheorem{lem}[thm]{Lemma}
\newtheorem{cor}[thm]{Corollary}

\theoremstyle{definition}
\newtheorem*{df}{Definition}

\theoremstyle{remark}
\newtheorem*{rem}{Remark}

\begin{document}


\title[Minimality and nonergodicity]{Minimality and nonergodicity on a family of flat surfaces in genus 3}

\author{Emanuel Nipper}
\address{Universit\"at Bonn, Mathematisches Institut, Beringstra\ss{}e 1, 53115 Bonn, Germany}
\email{emanuel@math.uni-bonn.de}





\begin{abstract}
We prove that a certain family of flat surfaces in genus $3$ does not fulfill Veech's Dichotomy. These flat surfaces provide uncountably many minimal but nonergodic directions. The conditions on this family are a combinatorical one and an irrationality condition. The Arnoux-Yoccoz surface fulfills this conditions.
\end{abstract}

\maketitle

\setcounter{section}{0}



\section{Flat surfaces}

Suppose $(X,\omega)$ is a flat surface, i.e.~$X$ is a Riemann surface and $\omega$ is an abelian differential (or holomorphic one-form) on $X$. Integrating this differential gives rise to an atlas, which in turn leads to a metric euclidean outside the zeros of $\omega$ and with cone type singularities with angle $2\pi (k+1)$ at a zero of order $k$. For a given angle $\vartheta \in S^1$, there exists a vector field, defined on the complement of the zeroes of $\omega$, such that the flow lines of the vector field are leafs of the horizontal foliation of $e^{-i\vartheta}\omega$. A flow line of this vector field is called a \textit{saddle connection} in direction $\vartheta$ if it joins two singularities and has no singularity in its interior. There are countably many saddle connections on a given flat surface.
\newline
The group $\SL$ acts on the moduli space of flat surfaces by post-composition with the charts given by integrating $\omega$. Let $\text{SL}(X,\omega)$ be the stabilizer under $\SL$ of the flat structure on $(X,\omega)$, i.e.~$\text{SL}(X,\omega) = \{A \in \SL \,:\,$ there is an affine linear diffeomorphism $f: X \to X$ such that in moduli space $A(X,\omega) = (X,f_* \omega)\}$. We call a flat surface a \textit{Veech surface} if $\text{SL}(X,\omega) \subset \SL$ is a lattice.
\newline
We may ask for the dynamical properties of the directional flow on a flat surface $(X,\omega)$. W.~A.~Veech proved that Veech surfaces fulfill the \textit{Veech dichotomy}, i.e.~for any $\vartheta \in S^1$ the flow in this direction is either periodic or minimal and uniquely ergodic, \cite{Vee89a}. It is known that in genus 2 this is an equivalence, see C.~McMullen's paper \cite{McM05a}. There are some examples by J.~Smillie and B.~Weiss of non-Veech surfaces that fulfill the Veech dichotomy in genus 5 and higher, \cite{SmiWei06b}. We will try to spread some light on what happens in genus 3.
\newline
\textit{Acknowledgement.} I would like to thank U.~Hamenst\"adt for numerous discussions and a careful reading as well as Y.~Cheung, P.~Hubert and M.~M\"oller for numerous discussions on this topic.

\section{A family of surfaces in genus 3}

In \cite{HubLanMol07a} and \cite{HubLanMol06a} P.~Hubert, E.~Lanneau and M.~M\"oller are dealing with a special family of flat surfaces: so called 2T2C-surfaces. By definition these are flat surfaces of genus $3$ that admit a direction $\vartheta$, such that the saddle connections in direction $\vartheta$ decompose the flat surface into two tori $T_1$ and $T_2$ and two cylinders $C_1$ and $C_2$. Figure \ref{fig:2t2c} shows a 2T2C-surface.
\begin{figure}
	\input{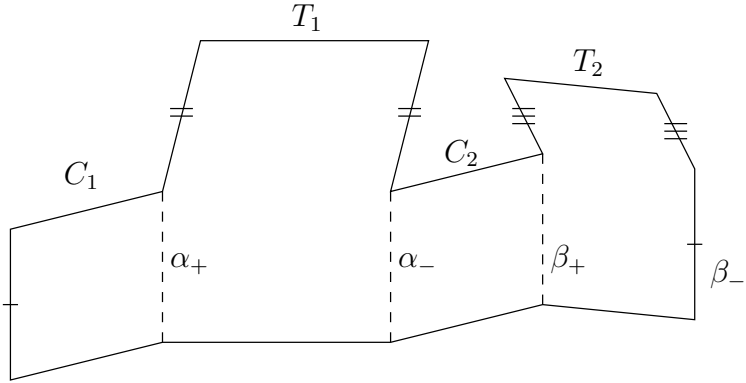}
	\caption{A 2T2C-splitting. Glueings are made as indicated and by vertical translations.}
	\label{fig:2t2c}
\end{figure}
\newline
Under certain circumstances, a 2T2C-surface is an element of the hyperelliptic locus $\Ls$ of the non-hyperelliptic component\footnote{More on connected components of the moduli space of abelian differentials can be found in Kontsevich's and Zorich's paper \cite{KonZor03a}.} $\Hs (2,2)^{odd}$, the set of all genus-3 abelian differentials with two zeroes of order two and with odd spin structure\footnote{We do not need properties of the spin structure explicitly. A definition can be found in, say, \cite{KonZor03a}.}. Namely, the condition for 2T2C-surfaces to be an element of the hyperelliptic locus can be phrased in the following way: If the two cylinders $C_1$ and $C_2$ represent the same class in the space of flat surfaces modulo isometries preserving the horizontal direction (\textit{marked isometries} for short), i.e.~if there is an isometry between $C_1$ and $C_2$ which maps horizontal lines to horizontal lines, then the flat surface is an element of $\Ls$. The same statement is true if the marked isometry classes of the two tori $T_1$ and $T_2$ coincide. According to \cite{HubLanMol06a} we will call a 2T2C-surface a \textit{\fix -surface} if the first condition holds, i.e.~the hyperelliptic involution fixes the two tori and exchanges the two cylinders. We call the according splitting a \textit{\fix -splitting}.
\newline
We formulate the main result of this note as a theorem:
\begin{thm} \label{thm_uncountable}
If a flat surface in genus 3 admits a \fix -splitting and if the direction of this splitting is nonperiodic in both tori, then there are uncountably many minimal but nonergodic directions on that surface.
\end{thm}

We record one application of the above theorem:

\begin{cor} \label{cor_arnoux-yoccoz}
The Arnoux-Yoccoz surface in genus $3$ admits uncountably many minimal, nonergodic directions.
\end{cor}

\begin{proof}
Let $(X,\omega)$ be the genus-$3$ Arnoux-Yoccoz surface. Hubert, Lanneau and M\"oller examined the Teichm\"uller disc of the Arnoux-Yoccoz surface in their paper \cite{HubLanMol07a}. They proved that $(X,\omega)$ admits a \fix -splitting, nonperiodic in both tori (this is Lemma 5.4, especially Claim 5.5). Using the above result we conclude that there are uncountably many minimal, nonergodic directions on $(X,\omega)$.
\end{proof}

\begin{rem}
More on the Arnoux-Yoccoz surface can be found in \cite{ArnYoc06a}.
\end{rem}

To prove Theorem \ref{thm_uncountable}, we follow the ideas of Y.~Cheung and H.~Masur in \cite{CheMas06a} very closely. Recall their ideas: Let $(\tilde{X},\tilde{\omega}) \in \Hf(2)$ be an L-shaped genus-2 flat surface, i.e.~we are given a splitting of $(\tilde{X},\tilde{\omega})$ into one torus $\tilde{T}$ and one cylinder $\tilde{C}$. Suppose that the direction of the splitting is minimal in $\tilde{T}$. Inductively applying Dehn twists along wisely chosen simple closed curves leads to new splittings meeting the minimality condition again. To find these simple closed curves, Cheung and Masur exploit a theorem of McMullen, which in turn is based on Ratner's theorem. Moreover, they are able to control the area exchanged between two splittings, the angle between the directions of two splittings and the heights of the vectors giving the splittings. The minimality condition is an important ingredient in this step. Using these pieces of information about the generated splittings, a theorem of Masur and Smillie leads to nonergodic directions.\\
We will use the same strategy. We just have to make sure that \fix -surfaces (which in some sense look like double-L surfaces) behave in a way very much the same like Cheung's and Masur's L-shaped surfaces do.\\
Sections \ref{section_twist}, \ref{section_induction} and \ref{section_uncountable} contain the proof of Theorem \ref{thm_uncountable}.
\newline

Let $(X,\omega) \in \Ls$ be a \fix -surface and let $\pi$ be projection $X \to \Cb\Pb^1$ coming from the hyperelliptic involution. Let $\Qf(1,1,(-1)^6)$ be the stratum of quadratic differentials on $\Cb\Pb^1$ with two simple zeros and six simple poles. On $\Cb\Pb^1$ there exists a unique quadratic differential $q$ such that $(\Cb\Pb^1,q) \in \Qf(1,1,(-1)^6)$ and $\pi^*q = \omega^2$. Conversely let $(\Cb\Pb^1,q) \in \Qf(1,1,(-1)^6)$ be given. Pulling back $q$ to the two-sheeted cover $\pi:X\to\Cb\Pb^1$, branched at the simple poles of $q$, we get an Abelian differential $\omega$ with $\pi^*q = \omega^2$ and $(X,\omega) \in \Ls$. This gives a local $\GL$-equivariant isomorphism between $\Ls$ and $\Qf(1,1,(-1)^6)$. In \cite{Lan05a} E.~Lanneau has shown that $\Qf(1,1,(-1)^6)$ and $\Qf(1,1,1,1)$, the principle stratum of quadratic differentials in genus 2, are locally $\GL$-equivariant isomorphic, thus studying $\Ls$ arises naturally as it is the next easy case beyond abelian differentials in genus $2$. Both strata $\Qf(1,1,(-1)^6)$ and $\Qf(1,1,1,1)$ have complex dimension 6, and so has $\Ls$. The paper \cite{Lan05a} contains more information on these strata, including further references to Kontsevich, Masur, Smillie, Veech and Zorich.
\newline
Let $u_1, \dotsc, u_6$ be saddle connections as shown in Figure \ref{fig:coordinates}.
\begin{figure}
	\input{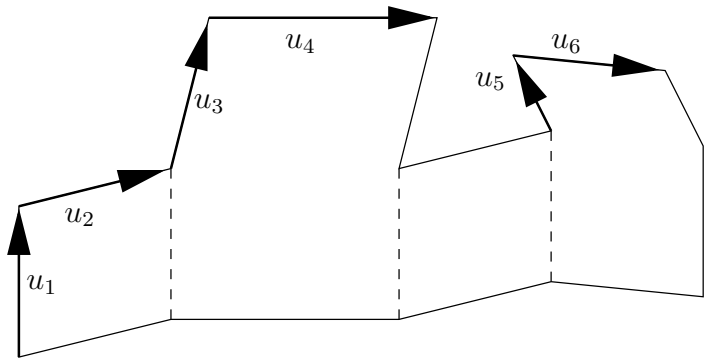}
	\caption{A six-tuple of saddle connections which serve as local coordinates.}
	\label{fig:coordinates}
\end{figure}
These saddle connections may serve as a set of local coordinates for a neighborhood of a \fix -surface $(X,\omega)$ in the hyperelliptic locus $\Ls$.
\newline
A result of H.~Masur and J.~Smillie (\cite{MasSmi91a}) states that in every stratum of quadratic differentials the Hausdorff dimension of the set of non-ergodic directions on a generic point is positive, hence there are uncountably many non-ergodic directions on a generic \fix -surface. Our result fits in this picture: Slightly deforming a \fix -surface in $\Ls$ (in other words: varying $u_1,\dotsc, u_6$ in an open neighborhood) does not destroy the property of being a \fix -surface, thus this property is an open condition and therefore the set of \fix -surfaces has positive (Lebesgue-) measure in $\Ls$. Hence \cite{MasSmi91a} gives uncountably many non-ergodic directions for almost all \fix -surfaces.
\newline
The prerequisites on the splitting-direction (i.e.~non-periodicity in the tori $T_1$ and $T_2$) locally rule out a countable union of real codimension-1-submanifolds: The direction of $u_1$ must not be the direction of any vector neither in the lattice spanned by $u_3$ and $u_4$ nor in the lattice spanned by $u_5$ and $u_6$ (see claim in proof of Corollary \ref{lem_twist_cheung}). For chosen $u_2, \dotsc, u_6$ both lattices exclude countably many directions for $u_1$. So we see that the set of admissible flat surfaces for our construction is locally the complement of a nullset. Thus we are in the generic case for \cite{MasSmi91a}.
\newline
In some sense Theorem \ref{thm_uncountable} actually gives more information than \cite{MasSmi91a}: we have a concrete description of the set of exceptional surfaces and therefore we can check whether a given surface meets our conditions or not, compare Corollary \ref{cor_arnoux-yoccoz}.


\section{Splittings and twisting a splitting} \label{section_twist}

Our task is to prove the theorem stated above. We will use an inductive process to construct the nonergodic directions. In this section we will collect some pieces of information on splittings. The following section is devoted to the inductive argument, which in turn will give the desired result in the last section. First, let us fix some notation.

\begin{df}
Consider a \fix -surface. Let $\alpha_+$, $\alpha_-$, $\beta_+$ and $\beta_-$ be the saddle connections that give the \fix-splitting into tori $T_1$, $T_2$ and cylinder $C_1$, $C_2$ as shown in Figure \ref{fig:2t2c}, i.e.~$\partial T_1 = \alpha_- - \alpha_+$ and $\partial T_2 = \beta_- - \beta_+$ (consider the geodesics to be running from bottom to top). Denote by $w$ the common holonomy of these saddle connections. Let $C$ be the common marked isometry class of $C_1$ and $C_2$. We will denote the splitting by $(T_1,T_2,C,w)$ for short.\\
The splitting is called irrational if the direction of $w$ is a minimal direction on both tori $T_1$ and $T_2$.\\
By $\Lambda_1$, $\Lambda_2$ and $\Lambda_c$ we will denote the lattices defining the tori $T_1$, $T_2$ and the marked isometry class $C$.
\end{df}

\begin{rem}
In \cite{HubLanMol07a}, a splitting is called irrational, if the direction is minimal in at least one torus. We need the stronger condition of nonperiodicity in both tori.
\end{rem}

\begin{rem}
In our notation, the hyperelliptic involution interchanges $\alpha_+$ and $\alpha_-$ as well as $\beta_+$ and $\beta_-$, see \cite{HubLanMol07a}.
\end{rem}

We need to speak about the oriented areas of parallelograms and about the areas of marked isometry classes of cylinders.

\begin{df}
Let $v \times w$ be the signed area of the parallelogram spanned by two vectors $v$ and $w$: the absolute value of $v \times w$ equals the euclidean area of the parallelogram and the sign is chosen to be positive if the pair $(v,w)$ is positively oriented, negative otherwise.\\
Let $\area(C)$ be the area of one and, hence, the common area of all representating cylinders of the marked isometry class $C$.
\end{df}

Given a tupel $(T_1,T_2,C,w)$ such that on a representating cylinder in $C$ there is a closed curve which length and direction equal the length and direction of $w$ and such that curves in direction of $w$ don't close up with length less or equal to $\abs{w}$ on $T_1$ and $T_2$, we can construct a pair $(X,\omega)$ of a genus-3 Riemann surfaces $X$ and an abelian differential $\omega$ on $X$ by slitting the tori $T_1$ and $T_2$ and two copies $C_1$ and $C_2$ of a cylinder in $C$ such that the holonomy of each slit equals $w$, and gluing the slitted surfaces according to the pattern shown in Figure \ref{fig:2t2c}. The flat structures on the tori and the cylinders coincide on the slits. The abelian differential $\omega$ on the resulting surface $X$ is given by these flat structures, compare K.~Strebel \cite[Paragraph 12.3]{Str84a}.
\newline
Let $v_1$, $v_2$ be the holonomy vectors of simple closed curves in $T_1$ and $T_2$, joining the initial point of the slit to itself, not intersecting the interior of the slit, and let $v_c := v_{c_1} = v_{c_2}$ be the common holonomy of curves in $C_1$ and $C_2$, joining the zero of $\omega$ on one boundary component of the cylinder to the second zero on the other boundary component. In Figure \ref{fig:2t2c}, these curves might be the bottom lines of the cylinders and tori, for instance. The parallelogram spanned by $v_j$ and $w$, $j \in \{1,2,c\}$, is isometrically embedded in the respective torus or cylinder, therefore $\abs{v_i \times w} \leq \area(T_i)$, $i \in \{1,2\}$, and $\abs{v_c \times w} \leq \area(C)$. If the simple closed curves concatenate to the core curve of a cylinder, they must have compatible orientations: $v_j \times w > 0$ for all $j \in \{1,2,c\}$ or $v_j \times w < 0$ for all $j \in \{1,2,c\}$.
\newline
Conversely, suppose that $v_1$, $v_2$, $v_c$ are three primitive vectors such that the conditions on area and orientation are satisfied. The area condition assures that on $C_1$ and $C_2$ there are curves with holonomy $v_c$, joining the zero of $\omega$ on one boundary component of the cylinder to the second zero on the other boundary component, and that on $T_j$, $j \in \{1,2\}$, there is a pair of simple closed curves with holonomy $v_j$ joining the initial/terminal point of the slit to itself. The orientation condition allows us to concatenate these curves. For given $k \in \Nb$, we get four new simple closed curves $\alpha_+^k$, $\alpha_-^k$, $\beta_+^k$ and $\beta_-^k$ by Dehn twisting $\alpha_+$, $\alpha_-$, $\beta_+$ and $\beta_-$ along the concatenated curve $k$ times. If we are lucky, each of the twisted curves can be realized by a single saddle connection (Lemma \ref{lem_twist_saddle}) with holonomy $w^k = w + k (v_1 + v_2 + 2v_c)$.\\
We adjust Cheung's and Masur's definition of good partners to our needs.

\begin{df}
We call the triple $(v_1,v_2,v_c)$ of holonomy vectors good partners (with respect to $w$), if
\begin{align*}
4 \abs{v_1 \times v_2} & <  \frac{1}{9} \min\{\abs{v_1 \times w}, \abs{v_2 \times w}\}, \\
4 \abs{v_1 \times v_c} & <  \frac{1}{9} \min\{\abs{v_1 \times w}, \abs{v_c \times w}\} \text{ and} \\
4 \abs{v_2 \times v_c} & <  \frac{1}{9} \min\{\abs{v_2 \times w}, \abs{v_c \times w}\}. \\
\end{align*}
\end{df}

We consider a pair $(X,\omega) \in \Ls$ that admits a \fix -splitting $(T_1,T_2,C,w)$. The following lemma answers the question, under which conditions a Dehn twist leads to another \fix -splitting.

\begin{lem} \label{lem_twist_saddle}
Each of the twisted simple closed curves $\alpha_+^k$, $\alpha_-^k$, $\beta_+^k$ and $\beta_-^k$ is realized by a single saddle connection if $w$ and $w^k$ lie on the same side of $v_1$, $v_2$ and $v_c$, i.e.~if all cross products $v_j \times w$ and $v_j \times w^k$, $j \in \{ 1,2,c \}$, are positive or all are negative.
\end{lem}

This lemma corresponds to Lemma 3.1 in \cite{CheMas06a}.

\begin{proof}
Suppose $v_j \times w > 0$ for $j \in \{ 1,2,c \}$. The inequality $v_j \times w^k > 0$ will be referred to as ($j$), $j \in \{ 1,2,c \}$. If $v_1 \times (v_2 + 2v_c) >0$ we consider $\alpha_+$, otherwise $\alpha_-$ (thanks to the hyperelliptic involution, the other $\alpha_\pm$ will be realized by a single saddle connection, too). So assume $v_1 \times (v_2 + 2v_c) >0$.
\begin{figure}
	\input{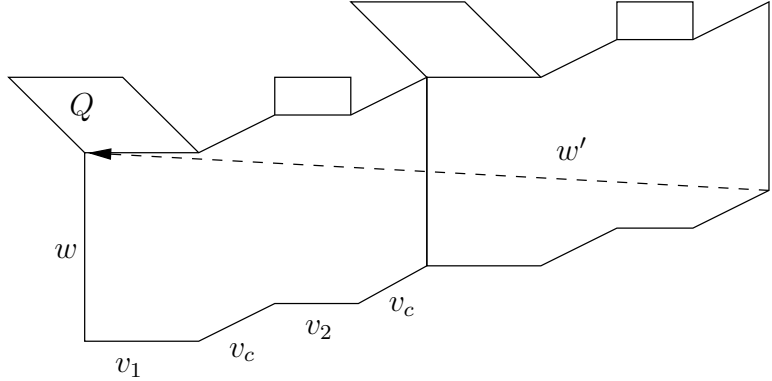}
	\caption{A splitting vector $w$ and its twist $w'$. The horizontal cylinder $Q$ is not affected.}
	\label{fig:twist}
\end{figure}
\newline
First, let $k <0$. As the $\SL$-action preserves cross products, we may assume that $w$ is vertical and $v_1$ is horizontal. The inequalities ($1$), ($2$) and ($c$) tell us that, as in Figure \ref{fig:twist}, the vector $w^k$ is above the lower boundary, hence the only possibility to hit (the image of) a singularity is by reaching (or by crossing) the upper boundary. Moreover, if $w^k$ crosses the upper boundary once, it will not come back from above. Let $\vrt(x)$ denote the vertical part of a vector $x$. We examine equation ($1$) more closely:
\begin{align*}
0 < v_1 \times w^k & = (v_1 \times w) + k (v_1 \times v_2) + 2k (v_1 \times v_c) \\
& = \frac{1}{\abs{v_1}} \vrt(w) + \frac{k}{\abs{v_1}} (\vrt(v_2) + 2\vrt(v_c)),
\end{align*}
therefore
\[ \vrt(w) > -k \vrt(v_2 + 2v_c). \]
This implies that $w^k$ does not cross the upper boundary and therefore does not hit a singularity beside at its endpoints. Hence, $\alpha_+$ is realized by a single saddle connection and so is $\alpha_-$.
\newline
Now let $k > 0$. After applying the $\SL$-action we may assume $w$ to be vertical and $v_2 + 2v_c$ to be horizontal. This causes $\vrt(v_1)$ to be negative since $v_1 \times (v_2 + 2v_c) >0$. Again, the inequalities ($1$), ($2$) and ($c$) tell us that $w^k$ is above the lower boundary, and, if it crosses the upper one, it will stay above. The inequalities ($2$) and ($c$) lead to 
\begin{align*}
0 < (v_2 + 2v_c) \times w^k & = ((v_2 + 2v_c) \times w) + k ((v_2 + 2v_c) \times v_1)\\
& = \frac{1}{\abs{v_2 + 2v_c}} \vrt(w) + \frac{k}{\abs{v_2 + 2v_c}} \vrt(v_1),
\end{align*}
thus 
\[ \vrt(w) > -k \vrt(v_1). \]
As above, $\alpha_+$ is realized by a single saddle connection and so is $\alpha_-$.
\newline
In a similar manner we conclude for $v_j \times w < 0$ and $v_j \times w^k < 0$, $j \in \{1,2,c\}$ and for $v_1 \times (v_2 + 2v_c) < 0$.
\newline
The same reasoning applies to $\beta_+$ and $\beta_-$ instead of $\alpha_+$ and $\alpha_-$.
\end{proof}

\begin{rem}
\begin{itemize}
\item The new splitting is a \fix -splitting again: the two cylinders in the new splitting are isometric as flat surfaces, and the isometry may be chosen to preserve the horizontal direction.

\item An easy computation shows that good partners fulfill the conditions for this lemma with $\abs{k} \leq 9$.
\end{itemize}
\end{rem}

Similar to Cheung and Masur, we want to get new irrational splittings from old ones by applying Dehn twists.

\begin{lem} \label{lem_twist_irrat}
If $(v_1,v_2,v_c)$ are good partners in an irrational \fix -splitting $(T_1,T_2,C,w)$, then at least one of the nine twists $w^k$ (with $k \in \{ 1, \cdots, 9\}$ or $k \in \{ -1, \cdots, -9\}$) leads to an irrational \fix -splitting.
\end{lem}

In order to prove the above lemma, we need the following

\begin{lem} \label{lem_twist_cheung}
Let $(T_1,T_2,C,w)$ be a splitting, such that the slope of $w$ is irrational in $T_1$, and let $v_1$, $v_2$ and $v_c$ be holonomy vectors such that each $\alpha_+^k$, $\alpha_-^k$, $\beta_+^k$ and $\beta_-^k$ is realized by one saddle connection and such that we have a \fix -splitting $(T_1^k,T_2^k,C^k,w^k)$ for at least three different $k = k_1$, $k_2$, $k_3$. Then at least one of the three splittings is irrational in the respective $T_1^k$.
\end{lem}

\begin{rem}
The lemma is symmetric with respect to $T_1$ and $T_2$ in the sense that one of the three splittings is irrational in $T_2^k$, too. Of course, this $k$ may be different from the $k$ that we got for $T_1$.
\end{rem}

This lemma is very much the same as Lemma 4.2 in Cheung's and Masur's article \cite{CheMas06a}, we only have to adjust their proof slightly.

\begin{proof}
Let $Q$ be a maximal cylinder in the direction of $v_1$ in $T_1$ that is disjoint from $w$, see Figure \ref{fig:twist}. Let $\gamma_0$ be a simple segment in $Q$ that concatenates with $\alpha_+$ to form a simple closed curve in $T_1$. Let $v_0 = \hol(\gamma_0)$. Then $\Lambda_1$, the lattice of $T_1$, is generated by $v_0 + w$ and $v_1$.\\
We claim: The vector $w$ is a scalar multiple of an element in $\Lambda_1$ (\textit{rational} for short), if and only if $w$ is rational in $\Lambda_0 = \left<v_0,v_1\right>$.\\
Indeed: For $a$,$b \in \Zb$ and $c \in \Rb_{\geq 0}$, $w = c(av_0 + bv_1)$  is equivalent to $(1+ac)w = c(a(v_0+w)+bv_1)$ with $(1+ac) \neq 0$ since $v_0+w$ and $v_1$ are linearly independent. This proves the claim.
\newline
Now, $T_1$ and $T_1^k$ share the same cylinder $Q$. Therefore, $w^k$ is rational in $\Lambda_1^k$, the lattice of $L_1^k$, if and only if $w^k$ is rational in $\Lambda_0$.\\
To prove the lemma, suppose that $(T_1^k,T_2^k,C^k,w^k)$ were rational in $T_1^k$ for all $k \in \{k_1,k_2,k_3\}$. Then the three $w^k$ are parallel to elements in $\Lambda_0$. Let $v^*=v_2+2v_c$ and $\Lambda_0^* = \left<w, v_1 + v^*\right>$. As $\Lambda_0$ and $\Lambda_0^*$ share the three directions $w^k$ which are not parallel to each other, $\Lambda_0$ and $\Lambda_0^*$ are isogenous (c.f. \cite{McM05a}, proof of Theorem 7.3) and they share all possible directions. Thus $w$ is parallel to an element in $\Lambda_0$ and $(T_1,T_2,C,w)$ is rational in $T_1$, a contradiction.
\end{proof}

Suppose that the conditions of Lemma \ref{lem_twist_irrat} are satisfied. The above lemma tells us that each of the three triples $(w^1,w^2,w^3)$, $(w^4,w^5,w^6)$ and $(w^7,w^8,w^9)$ contains one splitting that is irrational in the respective torus $T_1^k$. Call these vectors $\tilde{w}^1$, $\tilde{w}^2$ and $\tilde{w}^3$. Applying Lemma \ref{lem_twist_cheung} in $(T_1,T_2,C,w)$ again, this time with respect to the triple $(\tilde{w}^1,\tilde{w}^2,\tilde{w}^3)$ and to the torus $T_2$, we get at least one splitting, that is irrational in $T_2^k$, too. This proves Lemma \ref{lem_twist_irrat}.
\newline
The following proposition gives us some information about the area exchanged between two splittings when we apply such a twist (c.~f.~\cite[Lemma 3.3]{CheMas06a}).

\begin{prop} \label{lem_twist_area}
Let $(T_1,T_2,C,w)$ be a splitting of $(X,\omega)$ and let $(T_1',T_2',C',w')$ be obtained by twisting $k$ times. The change of area between the two tori indexed by $1$ is estimated by $\area(T_1 \Delta T_1') \leq  2\abs{v_1 \times w} + \abs{k}(\abs{v_1 \times v_2} + 2\abs{v_1 \times v_c})$. For $\area(T_2 \Delta T_2')$ the indices $1$ and $2$ change positions.
\end{prop}

\begin{proof}
Let $Q$ be as in the proof of Lemma \ref{lem_twist_cheung}. We have $Q \subset T_1 \cap T_1'$, hence $T_1 \Delta T_1' \subset (T_1 \setminus Q) \cup (T_1' \setminus Q)$. The areas on the right hand side are easily computed: $\area(T_1 \setminus Q) = \abs{v_1 \times w}$ and $\area(T_1' \setminus Q) = \abs{v_1 \times w'}$. Using $w' = w + k(v_1 + v_2 + 2v_c)$ the first statement follows.\\
For the second statement, we remark that the construction is symmetric in $T_1$ and $T_2$.
\end{proof}


\section{The inductive process} \label{section_induction}

We are going to apply an inductive process to find uncountably many nonergodic directions. In order to handle this inductive process we need an algebraic lemma that gives information about some orbit closures.
\newline
The following notations will be helpful: Let $G = \SL$ and let $N$ be the unipotent subgroup of upper triangular matrices. We will look at the action of $N$ on triples of unimodular lattices. This action is given by the diagonal group of $N$ which we denoted by $N_\Delta$. For $s \in \Rb$, let $G_s = \{(g, (n_s)^{-1} g n_s) \,:\, g \in G\}$ be the twisted diagonal of $G$, where $n_s = \genfrac{(}{)}{0pt}{}{1 \, s}{0 \, 1}$. Two lattice $\Lambda_1$ and $\Lambda_2$ are said to be strongly non-commensurable if there is no $s \in \Rb$ such that $\Lambda_1$ and $n_s \Lambda_2$ are commensurable.

\begin{prop} \label{induction_n-action}
Let $\Lambda_c$ be the standard lattice with area 1 and let $\Lambda_1$ and $\Lambda_2$ be two unimodular lattices, neither of them containing a horizontal vector. If $\Lambda_1$ and $\Lambda_2$ are strongly non-commensurable then $\overline{N_\Delta(\Lambda_1,\Lambda_2,\Lambda_c)} = (G \times G \times N)(\Lambda_1,\Lambda_2,\Lambda_c)$. Otherwise $\overline{N_\Delta(\Lambda_1,\Lambda_2,\Lambda_c)} = (G_s \times N)(\Lambda_1,\Lambda_2,\Lambda_c)$, where $s \in \Rb$ is such that $\Lambda_1$ and $n_s \Lambda_2$ are commensurable.
\end{prop}

\begin{proof}
The first case is Corollary 5.3 in \cite{HubLanMol06a}.
\newline
The second case can be proved as follows: Ratner's theorem tells us that we can write $\overline{N_\Delta(\Lambda_1,\Lambda_2,\Lambda_c)} = H(\Lambda_1,\Lambda_2,\Lambda_c)$ for some $H < G \times G \times G$. We note that $\pi_{1,2}(H) = G_s$ under the projection $\pi_{1,2}$ to the two first factors and $\pi_{1,3}(H) = G \times N$ under the projection $\pi_{1,3}$ to the first and third factor (\cite{McM07a}, Theorem 2.6), therefore $H < G_s \times N$. We need to show equality. Let $(g, (n_s)^{-1} g n_s,n) \in G_s \times N$. As the image of $H$ under $\pi_{1,3}$ equals $G \times N$ and as $(g,n) \in G \times N$, we know that $(g,g^*,n) \in H$ for some $g^* \in G$. The first projection gives $\pi_{1,2}((g,g^*,n)) \in G_s$, hence $g^*=(n_s)^{-1} g n_s$ and therefore $(g, (n_s)^{-1} g n_s,n) \in H$.
\end{proof}

We can handle the inductive process using this proposition. Again, we make use of Cheung's and Masur's ideas. Namely we will adopt their ideas of \cite[Proposition 4.6]{CheMas06a}.

\begin{prop} \label{inductive_new_splitting}
Given an irrational \fix -splitting $(T_1,T_2,C,w)$ and $\varepsilon > 0$, there exist two new irrational \fix -splittings with small change of direction $\angle(w,w') < \varepsilon$ and such that the area exchanged is smaller than $\varepsilon$.
\end{prop}

\begin{proof}
For convenience, suppose $w$ to be horizontal. Let $\Lambda_1$, $\Lambda_2$, $\Lambda_c$ denote the lattices of $T_1$, $T_2$, $C$. Let $\varepsilon' > 0$ be small. We will consider two cases.\\
First, let $\Lambda_1$ and $\Lambda_2$ be strongly non-commensurable. In this case $\overline{N_\Delta(\Lambda_1,\Lambda_2,\Lambda_c)} = (G \times G \times N)(\Lambda_1,\Lambda_2,\Lambda_c)$. Choose $v_c^* \in \Lambda_c^* \in N \cdot \Lambda_c$ with $\abs{v_c^* \times w} = \area(C)$ and choose $(v_1^*, v_2^*) \in (\Lambda_1^*, \Lambda_2^*) \in (G\times G)(\Lambda_1, \Lambda_2)$ with $\abs{v_i^*} < \min( \frac{\varepsilon'}{\abs{w}}, \frac{\area(T_i)}{\abs{w}}, \frac{\varepsilon'}{\abs{v_c^*}})$ for $i \in \{1,2\}$. Recall $\abs{x \times y} = \abs{\sin(\angle(x,y))} \cdot \abs{x}\abs{y}$, thus for $i \in \{1,2\}$ we have
\begin{itemize}
\item $\abs{v_i^* \times w} < \varepsilon'$
\item $\abs{v_i^* \times v_c^*} < \varepsilon'$
\item $\abs{v_1^* \times v_2^*} < \bigl( \frac{1}{\abs{w}} \bigr)^2 \bigl(\varepsilon'\bigr)^2 < \varepsilon'$
\end{itemize}
if only $\varepsilon'$ is small enough. We will refer to these inequalities as Properties (P). \\
We use the $(G \times G \times N)$-action to make $\angle(v_c^*,w)$ close to $\pi/2$. This will cause $\abs{v_c^*}$ to be close to $\frac{\area(C)}{\abs{w}}$. If necessary shorten the $v_i^*$ and in any case make $\angle(v_i^*,w)$ close to $\pi/2$, preserving the length of $v_i^*$, $i \in \{1,2\}$. Properties (P) are fulfilled again.
\newline
Look at the cross products $\abs{v_1^* \times v_2^*}$ and $\abs{v_i^* \times v_c^*}$, $i\in\{1,2\}$. The angle condition on $v_1^*$ and $v_2^*$ guarantees $\angle(v_1^*,v_2^*)$ to be small, hence $\abs{v_1^* \times v_2^*}$ is arbitrarily close to zero. On the other hand, $\abs{v_1^* \times w}$ and $\abs{v_2^* \times w}$ are arbitrarily close to $\abs{v_1^*}\abs{w}$ and $\abs{v_2^*}\abs{w}$, both greater than zero, hence $\abs{v_1^* \times v_2^*} < \varepsilon' \min(\abs{v_1^* \times w}, \abs{v_2^* \times w})$, which will be called Property (Q). In addition, $\abs{v_i^* \times v_c^*} = \abs{\sin (\angle(v_i^*,v_c^*))} \cdot \abs{v_i^*}\abs{v_c^*} \leq 2\abs{\sin( \angle(v_i^*,v_c^*))} \cdot \abs{v_i^*} \frac{\area(C)}{\abs{w}}$ is close to zero, too, and $\abs{v_c^* \times w} = \area(C)$. Therefore the inequalities $\abs{v_i^* \times v_c^*} < \varepsilon' \min(\abs{v_i^* \times w}, \abs{v_c^* \times w})$ hold for $i \in \{1,2\}$ (Properties (R)).\\
Before we proceed with the next step, we will consider the commensurability case. Let $\Lambda_1$ and $n_s \Lambda_2$ be commensurable for some $s \in \Rb$. We want to find $(v_1^*,v_2^*,v_c^*) \in (\Lambda_1^*,\Lambda_2^*,\Lambda_c^*) \in (G_s \times N)(\Lambda_1,\Lambda_2,\Lambda_c)$ with Properties (P), (Q) and (R). Suppose $v_1^*$, $n_s v_2^*$ and $v_c^*$ are parallel vectors. We are interested in almost horizontal vectors. The directions of $n_s v_2^*$ and $v_2^*$ are nearly the same in this case. Suppose that all the angles between $v_1^*$, $v_2^*$, $n_s v_2^*$, $v_c^*$ and $w$ are so small that the sines can be approximated by the angles within a multiplicative error less than 2. As a first calculation we get
\begin{align*}
\abs{v_2^* \times v_c^*} & = \abs{\sin(\angle(v_2^*,v_c^*))} \cdot \abs{v_2^*} \abs{v_c^*} \\
   & = \abs{\sin(\angle(v_2^*,v_c^*))} \cdot \abs{v_2^*} \frac{\area(C)}{\abs{w}} \frac{1}{\abs{\sin(\angle(v_c^*,w))}} \qquad \text{and}\\
\abs{v_2^* \times w} & = \abs{\sin(\angle(v_2^*,w))} \cdot \abs{v_2^*} \abs{w}.
\end{align*}
Let $x$ be the horizontal and $y$ be the vertical coordinate of $n_s v_2^*$. We compute the quotient $\frac{\abs{v_2^* \times v_c^*}}{\abs{v_2^* \times w}}$ and write the sines in terms of $x$ and $y$, using that sine is approximately linear for small angles:
\begin{align*}
\frac{\abs{v_2^* \times v_c^*}}{\abs{v_2^* \times w}} & = \frac{\area(C)}{\abs{w}^2} \frac{\abs{\sin(\angle(v_2^*,v_c^*))}}{\abs{\sin(\angle(v_2^*,w))} \cdot \abs{\sin(\angle(v_c^*,w))}}\\
   & \leq 4\frac{\area(C)}{\abs{w}^2} \frac{\left\lvert \frac{\abs{y}}{\sqrt{x^2+y^2}} - \frac{\abs{y}}{\sqrt{x^2 -2sxy +s^2 y^2}} \right\rvert}{\frac{\abs{y}}{\sqrt{x^2 -2sxy +s^2 y^2}} \frac{\abs{y}}{\sqrt{x^2+y^2}}}\\
   & = 4\frac{\area(C)}{\abs{w}^2} \left\lvert \sqrt{\frac{x^2}{y^2} - \frac{2sx}{y} + s^2} - \sqrt{\frac{x^2}{y^2} +1} \right\rvert\\
   & < \varepsilon'
\end{align*}
for $y < \frac{x}{K}$ with $K$ very large, i.e.~for $n_s v_2^*$ almost horizontal.\\
Choose $v_c^*$ almost horizontal, satisfying $\abs{v_c^* \times w} = \area(C)$. The $(G_s \times N)$-actions enables us to find a $v_2^*$ such that $v_c^*$ and $n_s v_2^*$ are parallel and the above inequality holds. Shortening $v_2^*$ without changing its direction assures $\abs{v_2^* \times v_c^*} < \varepsilon' \abs{v_c^* \times w}$, too. As $\Lambda_1$ and $n_s \Lambda_2$ are commensurable, the vectors $v_1^*$ and $n_s v_2^*$ can be chosen to be parallel, too, and Properties (R) are fulfilled after eventually shortening $v_1^*$. Again we use the group action to shorten $v_1^*$ and $v_2^*$ to make $\abs{v_1^* \times v_2^*} < \text{const} \cdot \abs{v_1^*}\abs{x}$ small compared to $\abs{v_1^* \times w} = \sin(\angle(v_1^*,w)) \cdot \abs{v_1^*}\abs{w}$ and compared to $\abs{v_2^* \times w} \geq \sin(\angle(v_2^*,w)) \cdot \abs{x}\abs{w}$, where the constant only depends on $K$. Property (Q) holds. For a last time we possibly have to shorten $v_1^*$ and $v_2^*$ to make Properties (P) being fulfilled.\\
In both cases -- commensurable and strongly non-commensurable lattices -- we approximate $(v_1^*,v_2^*,v_c^*)$ in $N_\Delta$-orbits of $(\Lambda_1,\Lambda_2,\Lambda_c)$ and note that $N_\Delta$ preserves cross products and leaves $w$ invariant. Hence, there is a $(v_1,v_2,v_c) \in (\Lambda_1,\Lambda_2,\Lambda_c)$ with Properties (P), (Q) and (R). For $\varepsilon'$ small enough, these are good partners and thus give rise to two irrational \fix -splittings (one for $k<0$ and one for $k>0$).\\
The lengths of vectors in a given lattice are bounded from below. Using the equality $\abs{\sin(\angle(x,y))} = \frac{\abs{x \times y}}{\abs{x}\abs{y}}$ we see that $\abs{\max(\angle(v_1,w), \angle(v_1,v_2), \angle(v_1,v_c))}$ is small and therefore we have $\angle(w,w') < \varepsilon$. Furthermore, using Proposition \ref{lem_twist_area}, $\area(T_1 \Delta T_1') < 2\abs{v_1 \times w} + \abs{k}(\abs{v_1 \times v_2} + 2\abs{v_1 \times v_c}) < \varepsilon$ and similarly $\area(T_2 \Delta T_2') \leq \varepsilon$ for small $\varepsilon'$.
\end{proof}


\section{Uncountably many nonergodic directions} \label{section_uncountable}

In this section we build a binary rooted tree such that the geodesics in the tree represent nonergodic directions on $(X,\omega)$. The elements of the tree are directions of irrational splittings of $(X,\omega)$. Our starting point is the direction $w$ of our initial splitting $(T_1,T_2,C,w)$. For every direction $w_n$ at level $n$, called parent, we will construct two different subsequent directions $w_n^1$ and $w_n^2$, called the children, such that these directions give rise to \fix -splittings with $\angle(w_n,w_n^i) < \frac{\varepsilon_n}{4}$ and $\area(T_1 \Delta T_1^i) < \frac{\varepsilon_n}{4}$, where $\varepsilon_n > 0$ is an arbitrary number smaller than all angles between any pair of parents and children constructed so far. In detail, given $w_n$ and $\varepsilon_n$ we apply Proposition \ref{inductive_new_splitting} with $\varepsilon < \frac{\varepsilon_n}{4}$ to find two different children $w_n^1$ and $w_n^2$ with the desired properties. The resulting tree contains $2^n$ directions of irrational splittings at its $n$-th level.\\
The geodesics in this tree represent different converging sequences of splitting directions: the angles between two subsequent directions converge to zero. The series of the changes of area will converge to a value smaller then the limit of the geometric series. We want to apply a lemma of Masur and Smillie and get an uncountable number of nonergodic directions:

\begin{lem} \label{tree_masur_smillie}
Let $(T_1^n,T_2^n,C^n,w_n)$ be a sequence of \fix -splittings of $(X,\omega)$ and assume that the directions of the vectors $w_n$ converge to some direction $\vartheta$. Let $h_n > 0$ be the component of $w_n$ perpendicular to $\vartheta$ and let $a_n$ bound the change of area from above: $a_n=\max(\area(T_1^n \Delta T_1^{n+1}),\area(T_2^n \Delta T_2^{n+1}))$. If
\begin{itemize}
\item $\sum_{n=1}^{\infty} a_n < \infty$,
\item there exists $c>0$ such that $\area(T_1^n) > c$, $\area(T_2^n) > c$ for all $n \in \Nb$ and
\item $\lim_{n \to \infty} h_n =0$.
\end{itemize}
then $\vartheta$ is a nonergodic direction.
\end{lem}

\begin{proof}
The proof can be found in \cite{MasSmi91a}, Theorem 2.1.
\end{proof}

We have to show that in every geodesic the heights $h_n$ perpendicular to the limiting direction of the splitting vectors $w_n$ converge to zero. To see this, we first compute $h_{n+1} \leq \frac{2\abs{w_n \times w_{n+1}}}{\abs{w_{n+1}}}$ using $\sin(\angle(x,y)) = \frac{\abs{x \times y}}{\abs{x}\abs{y}}$ and the fact that $\angle(w_n,\vartheta) \leq 2\angle(w_n,w_{n+1})$. Secondly, we note that the length of $w_n$ tends to infinity. Third ingredient is $\abs{w_n \times w_{n+1}} = \abs{w_n \times (w_n + k (v_1 + v_2 + 2v_c))} \leq 9 (\area(T_1) + \area(T_2) + 2\area(C))$. Combining these three facts, we see that $h_n$ converges to zero.

Thus, this construction gives an uncountable number of nonergodic directions. As there are at most countably many nonminimal directions, we can find uncountably many minimal and nonergodic directions. Theorem \ref{thm_uncountable} is proven.

\bibliographystyle{amsalpha}
\bibliography{/Users/emanuel/Documents/uni/Artikel/bib}

\end{document}